\documentclass[11pt,a4paper]{article}
\usepackage{pifont}
\usepackage{epsf,epsfig,amsfonts,amsgen,amsmath,amstext,amsbsy,amsopn,amsthm}
\usepackage{ebezier,eepic}
\usepackage{color}
\newtheorem{theorem}{Theorem}

\newtheorem{lemma}{Lemma}
\newtheorem{false statement}{False statement}

\theoremstyle{definition}

\newtheorem{problem}{Problem}

\newtheorem{claim1}{Claim}
\begin{document}

\title{\bf\Large Solution to a problem on hamiltonicity
of graphs under Ore- and Fan-type heavy subgraph
conditions\thanks{Supported by NSFC (No. 11271300), and the project NEXLIZ -- CZ.1.07/2.3.00/30.0038, which is
co-financed by the European Social Fund and the state budget of the
Czech Republic.}}

\date{}

\author{Bo Ning\footnote{Center for Applied Mathematics,~Tianjin University, Tianjin, 300072, P.R.~China. Email:
bo.ning@\break tju.edu.cn (B. Ning).}, Shenggui Zhang\footnote{Northwestern Polytechnical University, Xi'an, Shaanxi 710072, P.R.~China. Email:~sgzhang@nwpu.\break edu.cn (S. Zhang).}
and Binlong Li\footnote{Corresponding author. Northwestern Polytechnical University, Xi'an, Shaanxi 710072, P.R.~China;
Department of Mathematics, NTIS-New Technologies for the Information Society, University of West Bohemia, 30614 Pilsen, Czech Republic. Email: libinlong@mail.nwpu.edu.cn
(B. Li).}}

\maketitle
\begin{abstract}
A graph $G$ is called \emph{claw-o-heavy} if every induced claw
($K_{1,3}$) of $G$ has two end-vertices with degree sum at least
$|V(G)|$. For a given graph $S$, $G$ is called \emph{$S$-f-heavy} if
for every induced subgraph $H$ of $G$ isomorphic to $S$ and every
pair of vertices $u,v\in V(H)$ with $d_H(u,v)=2$, there holds
$\max\{d(u),d(v)\}\geq |V(G)|/2$. In this paper, we prove that every
2-connected claw-\emph{o}-heavy and $Z_3$-\emph{f}-heavy graph is
hamiltonian (with two exceptional graphs), where $Z_3$ is the graph
obtained by identifying one end-vertex of $P_4$ (a path with 4
vertices) with one vertex of a triangle. This result gives a
positive answer to a problem proposed in [B. Ning, S. Zhang, Ore-
and Fan-type heavy subgraphs for Hamiltonicity of 2-connected
graphs, Discrete Math. 313 (2013) 1715--1725], and also implies two
previous theorems of Faudree et al. and Chen et al., respectively.

\smallskip

\end{abstract}

\section{Introduction}
Throughout this paper, the graphs considered are simple, finite and
undirected. For terminology and notation not defined here, we refer
the reader to Bondy and Murty \cite{BM}.

Let $G$ be a graph. For a vertex $v\in V(G)$, we use $N_G(v)$ to
denote the set, and $d_G(v)$ the number, of neighbors of $v$ in $G$.
When there is no danger of ambiguity, we use $N(v)$ and $d(v)$
instead of $N_G(v)$ and $d_G(v)$. If $H$ and $H'$ are two subgraphs
of $G$, then we set $N_H(H')=\{v\in V(H): N_G(v)\cap
V(H')\neq\emptyset\}$. For two vertices $u,v\in V(H)$, the
\emph{distance} between $u$ and $v$ in $H$, denoted by $d_H(u,v)$,
is the length of a shortest path connecting $u$ and $v$ in $H$. In
particular, when we use the notation $G$ to denote a graph, then for
some subgraph $H$ of $G$, we set $N_H(v)=N_G(v)\cap V(H)$ and
$d_H(v)=|N_H(v)|$ (so, if $G'$ is another graph defined on the same
vertex set $V(G)$ and $H$ is a subgraph of $G'$, we will not use
$N_H(v)$ to denote $N_{G'}(v)\cap V(H)$).

We call $H$ an \emph{induced subgraph} of $G$, if for every $x,y\in
V(H)$, $xy\in E(G)$ implies that $xy\in E(H)$. For a given graph
$S$, $G$ is called $S$-\emph{free} if $G$ contains no induced
subgraph isomorphic to $S$. Following \cite{LRWZ}, $G$ is called
\emph{$S$-o-heavy} if every induced subgraph of $G$ isomorphic to
$S$ contains two nonadjacent vertices with degree sum at least
$|V(G)|$ in $G$. Following \cite{NZ}, $G$ is called
\emph{$S$-f-heavy} if for every induced subgraph $H$ isomorphic to
$S$ and any two vertices $u,v\in V(H)$ such that $d_H(u,v)=2$, there
holds $\max\{d(u),d(v)\}\geq |V(G)|/2$. Note that an $S$-free graph
is $S$-\emph{o}-heavy ($S$-\emph{f}-heavy).

The \emph{claw} is the bipartite graph $K_{1,3}$. Note that a
claw-\emph{f}-heavy graph is also claw-\emph{o}-heavy. Further
graphs that will be often considered as forbidden subgraphs are
shown in Fig. 1.

\begin{center}
\setlength{\unitlength}{0.9pt}
\begin{picture}(360,200)

\thicklines

\put(5,140){\multiput(20,30)(50,0){5}{\put(0,0){\circle*{6}}}
\put(20,30){\line(1,0){100}} \put(170,30){\line(1,0){50}}
\qbezier[4](120,30)(145,30)(170,30) \put(20,35){$v_1$}
\put(70,35){$v_2$} \put(120,35){$v_3$} \put(170,35){$v_{i-1}$}
\put(220,35){$v_i$} \put(115,10){$P_i$ (Path)}}

\put(265,130){\put(20,30){\circle*{6}} \put(70,30){\circle*{6}}
\put(45,55){\circle*{6}} \put(20,30){\line(1,0){50}}
\put(20,30){\line(1,1){25}} \put(70,30){\line(-1,1){25}}
\put(15,10){$C_3$ (Triangle)}}

\put(0,0){\put(20,30){\circle*{6}} \put(70,30){\circle*{6}}
\multiput(45,55)(0,25){4}{\put(0,0){\circle*{6}}}
\put(20,30){\line(1,0){50}} \put(20,30){\line(1,1){25}}
\put(70,30){\line(-1,1){25}} \put(45,55){\line(0,1){25}}
\put(45,105){\line(0,1){25}} \qbezier[4](45,80)(45,92.5)(45,105)
\put(50,80){$v_1$} \put(50,105){$v_{i-1}$} \put(50,130){$v_i$}
\put(40,10){$Z_i$}}

\put(90,0){\put(45,40){\circle*{6}} \put(45,40){\line(-1,1){25}}
\put(45,40){\line(1,1){25}} \put(20,65){\line(1,0){50}}
\multiput(20,65)(50,0){2}{\multiput(0,0)(0,30){2}{\put(0,0){\circle*{6}}}
\put(0,0){\line(0,1){30}}} \put(25,10){$B$ (Bull)}}

\put(180,0){\multiput(20,30)(50,0){2}{\multiput(0,0)(0,30){2}{\put(0,0){\circle*{6}}}
\put(0,0){\line(0,1){30}}}
\multiput(45,85)(0,30){2}{\put(0,0){\circle*{6}}}
\put(45,85){\line(0,1){30}} \put(20,60){\line(1,0){50}}
\put(20,60){\line(1,1){25}} \put(70,60){\line(-1,1){25}}
\put(25,10){$N$ (Net)}}

\put(270,0){\put(45,30){\circle*{6}} \put(20,55){\line(1,0){50}}
\put(45,30){\line(1,1){25}} \put(45,30){\line(-1,1){25}}
\multiput(20,55)(0,30){2}{\put(0,0){\circle*{6}}}
\multiput(70,55)(0,30){3}{\put(0,0){\circle*{6}}}
\put(20,55){\line(0,1){30}} \put(70,55){\line(0,1){60}}
\put(10,10){$W$ (Wounded)}}

\end{picture}

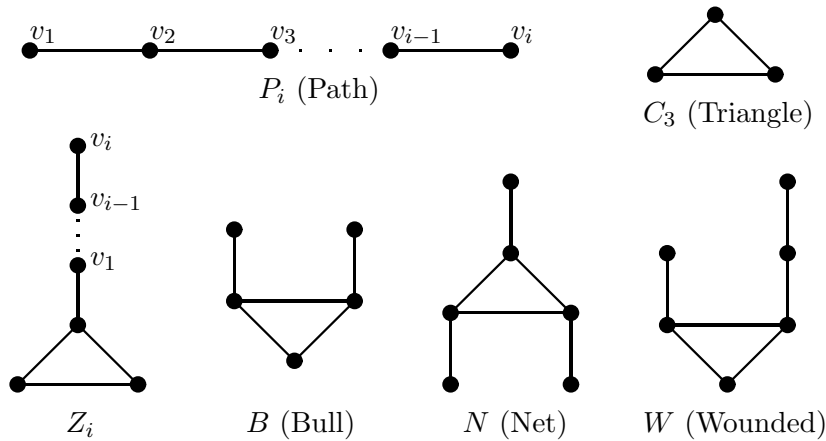
\small Fig.~1. Graphs $P_i,C_3,Z_i,B,N$ and $W$.
\end{center}

Bedrossian \cite{B} characterized all connected forbidden pairs for
a 2-connected graph to be hamiltonian.

\begin{theorem}\label{thB}
\emph{(Bedrossian \cite{B})} Let $G$ be a 2-connected graph and let
$R$ and $S$ be connected graphs other than $P_3$. Then $G$ being
$R$-free and $S$-free implies $G$ is hamiltonian if and only if (up
to symmetry) $R=K_{1,3}$ and $S=C_3,P_4,P_5,P_6,Z_1,Z_2,B,N$ or $W$.
\end{theorem}

Faudree and Gould \cite{FG} extended Bedrossian's result by giving a
proof of the `only if\,' part based on infinite families of
non-hamiltonian graphs.

\begin{theorem}\label{thFG}
\emph{(Faudree and Gould \cite{FG})} Let $G$ be a 2-connected graph
of order at least 10 and let $R$ and $S$ be connected graphs other
than $P_3$. Then $G$ being $R$-free and $S$-free implies $G$ is
hamiltonian if and only if (up to symmetry) $R=K_{1,3}$ and
$S=C_3,P_4,P_5,P_6,Z_1,Z_2,Z_3,B,N$ or $W$.
\end{theorem}

Li et al. \cite{LRWZ} extended Bedrossian's result by restricting
Ore's condition to pairs of induced subgraphs of a graph. Ning and
Zhang \cite{NZ} gave another extension of Bedrossian's theorem by
restricting Ore's condition to induced claws and Fan's condition to
other induced subgraphs of a graph.

\begin{theorem}\label{thNZ}
\emph{(Ning and Zhang \cite{NZ})} Let $G$ be a 2-connected graph and
$S$ be a connected graph other than $P_3$. Suppose that $G$ is
claw-o-heavy. Then $G$ being $S$-f-heavy implies $G$ is hamiltonian
if and only if $S=P_4,P_5,P_6,Z_1,Z_2,B,N$ or $W$.
\end{theorem}

Motivated by Theorems \ref{thFG} and \ref{thNZ}, Ning and Zhang
\cite{NZ} proposed the following problem.
\begin{problem}
(Ning and Zhang \cite{NZ}) Is every claw-\emph{o}-heavy and
$Z_3$-\emph{f}-heavy graph of order at least 10 hamiltonian?
\end{problem}

The main goal of this paper is to give an affirmative solution to
this problem. Our answer is the following theorem, where the graphs
$L_1$ and $L_2$ are shown in Fig.~2.

\begin{theorem}\label{thNZL}
Let $G$ be a 2-connected graph. If $G$ is claw-o-heavy and
$Z_3$-f-heavy, then $G$ is either hamiltonian or isomorphic to $L_1$
or $L_2$.
\end{theorem}

\begin{center}
\thicklines
\begin{picture}(280,120)
\put(0,0){\multiput(20,30)(80,0){2}{\put(0,0){\circle*{6}}
\put(0,80){\circle*{6}} \put(20,40){\circle*{6}}
\put(0,0){\line(0,1){80}} \put(0,0){\line(1,2){20}}
\put(0,80){\line(1,-2){20}}} \put(20,30){\line(1,0){80}}
\put(20,30){\line(2,1){40}} \put(100,30){\line(-2,1){40}}
\put(20,110){\line(1,0){80}} \put(20,110){\line(2,-1){40}}
\put(100,110){\line(-2,-1){40}} \put(60,50){\circle*{6}}
\put(60,90){\circle*{6}} \put(80,70){\circle*{6}}
\put(60,50){\line(0,1){40}} \put(60,50){\line(1,1){20}}
\put(60,90){\line(1,-1){20}} \put(65,10){$L_1$}}

\put(140,0){\multiput(20,30)(80,0){2}{\put(0,0){\circle*{6}}
\put(0,80){\circle*{6}} \put(20,40){\circle*{6}}
\put(0,0){\line(1,2){20}} \put(0,80){\line(1,-2){20}}}
\put(20,30){\line(1,0){80}} \put(100,30){\line(0,1){80}}
\put(20,30){\line(2,1){40}} \put(100,30){\line(-2,1){40}}
\put(20,110){\line(1,0){80}} \put(20,110){\line(2,-1){40}}
\put(100,110){\line(-2,-1){40}} \put(60,50){\circle*{6}}
\put(60,90){\circle*{6}} \put(80,70){\circle*{6}}
\put(60,50){\line(0,1){40}} \put(60,50){\line(1,1){20}}
\put(60,90){\line(1,-1){20}} \put(65,10){$L_2$}}
\end{picture}

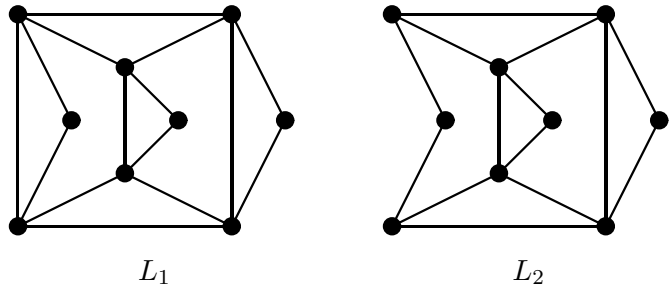
{\small Fig.~2. Graphs $L_1$ and $L_2$.}
\end{center}

Theorem \ref{thNZL} extends the following two previous theorems.

\begin{theorem}
\emph{(Faudree et al. \cite{FGRS})} If $G$ is a 2-connected
claw-free and $Z_3$-free graph, then $G$ is either hamiltonian or
isomorphic to $L_1$ or $L_2$.
\end{theorem}

\begin{theorem}
\emph{(Chen et al. \cite{CWZ})} If $G$ is a 2-connected claw-f-heavy
and $Z_3$-f-heavy graph, then $G$ is either hamiltonian or
isomorphic to $L_1$ or $L_2$.
\end{theorem}

We remark that there are infinite 2-connected claw-\emph{o}-heavy
and $Z_3$-\emph{o}-heavy graphs which are non-hamiltonian, see
\cite{LRWZ}.

Together with Theorem \ref{thNZ} and Theorem \ref{thNZL}, we can
obtain the following result which generalizes Theorem \ref{thFG}.

\begin{theorem}
Let $G$ be a 2-connected graph of order at least 10 and $S$ be a
connected graph other than $P_3$. Suppose that $G$ is claw-o-heavy.
Then $G$ being $S$-f-heavy implies $G$ is hamiltonian if and only if
$S=P_4,P_5,P_6,Z_1,Z_2,Z_3,B,N$ or $W$.
\end{theorem}
\section{Preliminaries}
In this section, we will list some necessary preliminaries. First,
we will introduce the closure theory of claw-\emph{o}-heavy graphs
proposed by \v{C}ada \cite{C}, which is an extension of the closure
theory of claw-free graphs due to Ryj\'a\v{c}ek \cite{R}.

Let $G$ be a graph of order $n$. A vertex $x\in V(G)$ is called
\emph{heavy} if $d(x)\geq n/2$; otherwise, it is called
\emph{light}. A pair of nonadjacent vertices $\{x,y\}\subset V(G)$
is called a \emph{heavy pair} of $G$ if $d(x)+d(y)\geq n$.

Let $G$ be a graph and $x\in V(G)$. Define $B^o_x(G)=\{uv:
\{u,v\}\subset N(x),~d(u)+d(v)\geq |V(G)|\}$. Let $G^o_x$ be a graph
with vertex set $V(G^o_x)=V(G)$ and edge set $E(G^o_x)=E(G)\cup
B^o_x(G)$. Suppose that $G^o_x[N(x)]$ consists of two disjoint
cliques $C_1$ and $C_2$. For a vertex $y\in
V(G)\backslash(N(x)\cup\{x\})$, if $\{x,y\}$ is a heavy pair in $G$
and there are two vertices $x_1\in C_1$ and $x_2\in C_2$ such that
$x_1y,x_2y\in E(G)$, then $y$ is called a \emph{join vertex} of $x$
in $G$. If $N(x)$ is not a clique and $G^o_x[N(x)]$ is connected, or
$G^o_x[N(x)]$ consists of two disjoint cliques and there is some
join vertex of $x$, then the vertex $x$ is called an
\emph{o-eligible vertex} of $G$. The \emph{locally completion of $G$
at $x$}, denoted by $G'_x$, is the graph with vertex set
$V(G'_x)=V(G)$ and edge set $E(G'_x)=E(G)\cup\{uv: u,v\in N(x)\}$.

Let $G$ be a claw-\emph{o}-heavy graph. The \emph{closure} of $G$,
denoted by
$cl_o(G)$, is the graph such that:\\
(1) there is a sequence of graphs $G_1,G_2,\ldots,G_t$ such that
$G=G_1$, $G_t=cl_o(G)$, and for any $i\in \{1,2,\ldots,t-1\}$, there
is an \emph{o}-eligible vertex $x_i$ of $G_i$, such that
$G_{i+1}=(G_i)'_{x_i}$; and\\
(2) there is no \emph{o}-eligible vertex in $G_t$.

\begin{theorem}\label{ThC}
\emph{(\v{C}ada \cite{C})} Let $G$ be a claw-o-heavy graph. Then\\
\emph{(1)} the closure $cl_o(G)$ is uniquely determined;\\
\emph{(2)} there is a $C_3$-free graph $H$ such that $cl_o(G)$ is
the line
graph of $H$; and\\
\emph{(3)} the circumferences of $cl_o(G)$ and $G$ are equal.
\end{theorem}

Now we introduce some new terminology and notations. Let $G$ be a
claw-\emph{o}-heavy graph and $C$ be a maximal clique of $cl_o(G)$.
We call $G[C]$ a \emph{region} of $G$. For a vertex $v$ of $G$, we
call $v$ an \emph{interior vertex} if it is contained in only one
region, and a \emph{frontier vertex} if it is contained in two
distinct regions. For two vertices $u,v\in V(G)$, we say $u$ and $v$
are \emph{associated} if $u,v$ are contained in a common region of
$G$; {otherwise $u$ and $v$ are \emph{dissociated}.} For a region
$R$ of $G$, we denote by $I_R$ the set of interior vertices of $R$,
and by $F_R$ the set of frontier vertices of $R$.

From the definition of the closure, it is not difficult to get the
following lemma.

\begin{lemma}\label{LeClosed}
Let $G$ be a claw-o-heavy graph. Then\\
\emph{(1)} every vertex is either an interior vertex of a region or
a frontier vertex of two regions;\\
\emph{(2)} every two regions are either disjoint or have only one
common
vertex; and\\
\emph{(3)} every pair of dissociated vertices have degree sum less
than $|V(G)|$ in $cl_o(G)$ (and in $G$).
\end{lemma}

\begin{proof}
In the proof of the lemma, we let $G'=cl_o(G)$.

(1) Let $v$ be an arbitrary vertex of $G$. Since $G'$ is closed,
$N_{G'}(v)$ is either a clique or a disjoint union of two cliques in
$G'$. Thus $v$ is contained in one or two regions of $G$, and the
assertion is true.

(2) Let $R$ and $R'$ be two regions of $G$, and $C$ and $C'$ be the
two maximal cliques of $G'$ corresponding to $R$ and $R'$,
respectively. If $C$ and $C'$ have two common vertices, say $u$ and
$v$, then $u$ and $v$ will be $o$-eligible vertices of $G'$,
contradicting the definition of the closure of $G$. This implies
that $C$ and $C'$ (and then, $R$ and $R'$) have at most one common
vertex.

(3) Let $u,v$ be two nonadjacent vertices with
$d_{G'}(u)+d_{G'}(v)\geq n=|V(G)|$. Then $u,v$ have at least two
common neighbors in $G'$. Suppose that $u$ and $v$ are not in a
common clique of $G'$. Let $x$ be a common neighbor of $u$ and $v$
in $G'$. Since $N_{G'}(x)$ is not a clique in $G'$, it is the
disjoint union of two cliques, one containing $u$ and the other
containing $v$. Since $uv\in B_x^o(G')$, $x$ is an $o$-eligible
vertex of $G'$, a contradiction. Thus we conclude that $u,v$ are in
a common clique of $G'$, i.e., $u$ and $v$ are associated.\qed
\end{proof}

The next lemma provides some structural information on regions.

\begin{lemma}\label{LeRegion}
Let $G$ be a claw-o-heavy graph and $R$ be a region of $G$. Then\\
\emph{(1)} $R$ is nonseparable;\\
\emph{(2)} if $v$ is a frontier vertex of $R$, then $v$ has an
interior
neighbor in $R$ or $R$ is complete and has no interior vertices;\\
\emph{(3)} for any two vertices $u,v\in R$, there is an induced path
of $G$ from $u$ to $v$ such that every internal vertex of the path
is an
interior vertex of $R$; and\\
\emph{(4)} for two vertices $u,v$ in $R$, if $\{u,v\}$ is a heavy
pair of $G$, then $u,v$ have two common neighbors in $I_R$.
\end{lemma}

\begin{proof}
Let $G_1,G_2,\ldots,G_t$ be the sequence of graphs, and
$x_1,x_2,\ldots,x_{t-1}$ the sequence of vertices in the definition
of $cl_o(G)$.

(1) Suppose that $R$ has a cut-vertex $y$. We prove by induction
that $y$ would be a cut-vertex of $G_i[V(R)]$ for all $i\in[1,t]$.
Since $y$ is a cut-vertex of $G_1[V(R)]=R$, we assume that $2\leq
i\leq t$. By the induction hypothesis, $y$ is a cut-vertex of
$G_{i-1}[V(R)]$. Let $R'$ and $R''$ be two components of
$G_{i-1}[V(R)]-y$, $u$ be a vertex of $R'$ and $v$ be a vertex of
$R''$. Then $u$ and $v$ have at most one common neighbor $y$ in $R$.
Note that each two maximal cliques of $cl_o(G)$ is either disjoint
or have only one common vertex (see Lemma \ref{LeClosed} (1)). This
implies that $u$ and $v$ have no common neighbors in $G_{i-1}-V(R)$.
Hence $\{u,v\}$ is not a heavy pair of $G$. Note that an
\emph{o}-eligible vertex of $G_{i-1}$ will be an interior vertex of
$cl_o(G)$. This implies that $y$ is not an \emph{o}-eligible vertex
of $G_{i-1}$. Thus $x_{i-1}\neq y$. Note that $x_{i-1}$ has no
neighbors in $R'$ or has no neighbors in $R''$. This implies that
there are no new edges in $G_i$ between $R'$ and $R''$. Thus $y$ is
also a cut-vertex of $G_i[V(R)]$. By induction, we can see that $y$
is a cut-vertex of $cl_o(G)[V(R)]$, contradicting the fact that
$V(R)$ is a clique in $cl_o(G)$.

(2) Note that $cl_o(G)[V(R)]$ is complete. If $R$ has no interior
vertex, then $R$ contains no \emph{o}-eligible vertex of $G$. Since
the locally completion of $G$ at every \emph{o}-eligible vertex does
not add an edge in $R$, $R=cl_o(G)[V(R)]$ is complete.

Now we assume that $R$ has at least one interior vertex. Suppose
that $v$ has no interior neighbors in $R$, i.e., $N(v)\cap
I_R=\emptyset$. Using induction, we will prove that $N_{G_i}(v)\cap
I_R=\emptyset$. Since $N_{G_1}(v)\cap I_R=\emptyset$, we assume that
$2\leq i\leq t$. By the induction hypothesis, $N_{G_{i-1}}(v)\cap
I_R=\emptyset$. Note that $x_{i-1}$ is either nonadjacent to $v$ or
nonadjacent to every vertex in $N_{G_{i-1}}(v)\cap V(R)$. This
implies that there are no new edges of $G_i$ between $v$ and
$G_i[V(R)]-v$. Hence $N_{G_i}(v)\cap I_R=\emptyset$. Thus by the
induction hypothesis, we can see that $N_{cl_o(G)}(v)\cap
I_R=\emptyset$, a contradiction.

(3) We use induction on $t-i$ ($t$ is the subscript of
$G_t=cl_o(G)$) to prove that there is an induced path of $G_i[V(R)]$
from $u$ to $v$ such that every internal vertex of the path is an
interior vertex of $R$. Note that $uv$ is an edge in $G_t[V(R)]$. We
are done if $i=t$. Now suppose that there is an induced path $P$ of
$G_i[V(R)]$ from $u$ to $v$ such that every internal vertex of the
path is an interior vertex of $R$. We will prove that there is an
induced path of $G_{i-1}[V(R)]$ from $u$ to $v$ such that every
internal vertex of the path is an interior vertex of $R$. If $P$ is
also a path of $G_{i-1}[V(R)]$, then we are done. So we assume that
there is an edge $u'v'\in E(P)$ such that $u'v'\notin E(G_{i-1})$.
This implies that $u',v'\in N(x_{i-1})$. Since $P$ is an induced
path of $G_i$, $x_{i-1}$ has the only two neighbors $u',v'$ on $P$.
We also note that $x_{i-1}\in V(R)$ is an interior vertex. Thus
$P'=(P-u'v')\cup u'xv'$ (with the obvious meaning) is an induced
path of $G_{i-1}[V(R)]$ from $u$ to $v$ such that every internal
vertex of the path is an interior vertex of $R$. Thus by the
induction hypothesis, the proof is complete.

(4) Since every vertex in $F_R$ has at least one neighbor in $G-R$
and every vertex in $G-R$ has at most one neighbor in $F_R$, we have
$|N_{G-R}(F_R\backslash\{u,v\})|\geq|F_R\backslash\{u,v\}|$.
Furthermore, we have
$n=|I_R\backslash\{u,v\}|+|F_R\backslash\{u,v\}|+|V(G-R)|+2$. Thus,
we get
\begin{align*}
n   & \leq d(u)+d(v)\\
    & =d_{I_R}(u)+d_{I_R}(v)+d_{F_R}(u)+d_{F_R}(v)+d_{G-R}(u)+d_{G-R}(v)\\
    & \leq d_{I_R}(u)+d_{I_R}(v)+2|F_R\backslash\{u,v\}|+
        d_{G-R}(u)+d_{G-R}(v)\\
    & \leq d_{I_R}(u)+d_{I_R}(v)+|F_R\backslash\{u,v\}|+
        |N_{G-R}(F_R\backslash\{u,v\})|+|N_{G-R}(u)|+|N_{G-R}(v)|\\
    & =d_{I_R}(u)+d_{I_R}(v)+|F_R\backslash\{u,v\}|+|N_{G-R}(F_R)|\\
    & \leq d_{I_R}(u)+d_{I_R}(v)+|F_R\backslash\{u,v\}|+|V(G-R)|,
\end{align*}
and $$d_{I_R}(u)+d_{I_R}(v)\geq
n-|F_R\backslash\{u,v\}|-|V(G-R)|=|I_R\backslash\{u,v\}|+2.$$ This
implies that $u,v$ have two common neighbors in $I_R$.\qed
\end{proof}

Let $G$ be a graph and $Z$ be an induced copy of $Z_3$ in $G$. We
denote the vertices of $Z$ as in Fig. 3, and say that $Z$ is
\emph{center-heavy} in $G$ if $a_1$ is a heavy vertex of $G$. If
every induced copy of $Z_3$ in $G$ is center-heavy, then we say that
$G$ is \emph{$Z_3$-center-heavy}.

\begin{center}
\begin{picture}(155,70)
\thicklines

\put(0,-20){\put(20,30){\circle*{5}} \put(20,80){\circle*{5}}
\put(20,30){\line(0,1){50}} \put(20,30){\line(1,1){25}}
\put(20,80){\line(1,-1){25}} \multiput(45,55)(30,0){4}{\circle*{5}}
\put(45,55){\line(1,0){90}} \put(8,78){$b$} \put(8,28){$c$}
\put(42,62){$a$} \put(72,62){$a_1$} \put(102,62){$a_{2}$}
\put(132,62){$a_3$}}

\end{picture}

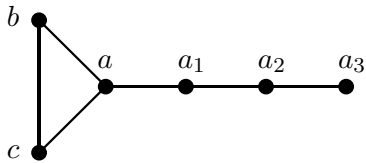
{\small Fig.~3. The Graph $Z_3$.}
\end{center}

\begin{lemma}\label{LeCenter}
Let $G$ be a claw-o-heavy and $Z_3$-f-heavy graph. Then $cl_o(G)$ is
$Z_3$-center-heavy.
\end{lemma}

\begin{proof}
Let $Z$ be an arbitrary induced copy of $Z_3$ in $G'=cl_o(G)$. We
denote the vertices of $Z$ as in Fig. 3, and will prove that $a_1$
is heavy in $G'$.

Let $R$ be the region of $G$ containing $\{a,b,c\}$. Recall that
$I_R$ is the set of interior vertices of $R$, and $F_R$ is the set
of frontier vertices of $R$.

\begin{claim1}\label{ClNeighbor}
$|N_R(a_2)\cup N_R(a_3)|\leq 1$.
\end{claim1}

\begin{proof}
Note that every vertex in $G-R$ has at most one neighbor in $R$. If
$N_R(a_2)=\emptyset$, then the assertion is obviously true. Now we
assume that $N_R(a_2)\neq\emptyset$. Let $x$ be the vertex in
$N_R(a_2)$. Clearly $x\neq a$ and $a_1x\notin E(G')$. If $a_3x\notin
E(G')$, then $\{a_2,a_1,a_3,x\}$ induces a claw in $G'$, a
contradiction. This implies that $a_3x\in E(G')$, and $x$ is the
unique vertex in $N_{G'}(a_3)\cap V(R)$. Thus $N_R(a_2)\cup
N_R(a_3)=\{x\}$.\qed
\end{proof}

\begin{claim1}\label{ClCommon}
Let $x,y$ be two vertices in $I_R\cup\{a\}$. If $xy\in E(G)$ and
$d(x)+d(y)\geq n$, then $x,y$ have a common neighbor in $I_R$.
\end{claim1}

\begin{proof}
Note that every vertex in $F_R$ has at least one neighbor in $G-R$
and every vertex in $G-R$ has at most one neighbor in $R$. By Claim
\ref{ClNeighbor}, $|V(G-R)|\geq|F_R|+1$. Moreover, since $a$ is not
the neighbor of $a_2$ and $a_3$ in $R$,
$|V(G-R)|\geq|F_R\backslash\{a\}|+|N_{G-R}(a)|+1$.

If $x,y\in I_R$, then
\begin{align*}
n   & \leq d(x)+d(y)\\
    & =d_{I_R}(x)+d_{I_R}(y)+d_{F_R}(x)+d_{F_R}(y)\\
    & \leq d_{I_R}(x)+d_{I_R}(y)+2|F_R|\\
    & \leq d_{I_R}(x)+d_{I_R}(y)+|F_R|+|V(G-R)|-1,
\end{align*}
and $$d_{I_R}(x)+d_{I_R}(y)\geq n-|F_R|-|V(G-R)|+1=|I_R|+1.$$ This
implies that $x,y$ have a common neighbor in $I_R$.

If one of $x,y$, say $y$ is $a$, then
\begin{align*}
n   & \leq d(x)+d(a)\\
    & =d_{I_R}(x)+d_{I_R}(a)+d_{F_R}(x)+d_{F_R}(a)+d_{G-R}(a)\\
    & \leq d_{I_R}(x)+d_{I_R}(a)+|F_R|+|F_R\backslash\{a\}|+d_{G-R}(a)\\
    & \leq d_{I_R}(x)+d_{I_R}(a)+|F_R|+|V(G-R)|-1,
\end{align*}
and $$d_{I_R}(x)+d_{I_R}(a)\geq n-|F_R|-|V(G-R)|+1=|I_R|+1.$$ This
implies that $x,a$ have a common neighbor in $I_R$.\qed
\end{proof}

By Lemma \ref{LeRegion} (3), $G$ has an induced path $P$ from $a$ to
$a_3$ such that every vertex of $P$ is either in $\{a,a_1,a_2,a_3\}$
or an interior vertex outside $R$. Let $a,a'_1,a'_2,a'_3$ be the
first four vertices of $P$.

Note that $a'_1$ is either $a_1$ or an interior vertex in the region
containing $\{a,a_1\}$. This implies that $d_{G'}(a_1)\geq
d_{G'}(a'_1)\geq d(a'_1)$. If $a'_1$ is heavy in $G$, then $a_1$ is
heavy in $G'$ and we are done. So we assume that $a'_1$ is not heavy
in $G$.

If $abca$ is also a triangle in $G$, then the subgraph induced by
$\{a,b,c,a'_1,a'_2,a'_3\}$ is a $Z_3$. Since $G$ is
$Z_3$-\emph{f}-heavy and $a'_1$ is not heavy in $G$, $b$ and $a'_3$
are heavy in $G$. By Lemma \ref{LeClosed} (3), $b$ and $a'_3$ are
associated, a contradiction. Thus we conclude that one edge of
$\{ab,ac,bc\}$ is not in $E(G)$.

Note that $R$ is not complete. By Lemma \ref{LeRegion} (2), $a$ has
a neighbor in $I_R$.

\begin{claim1}\label{ClInterior}
$d_{I_R}(a)=1$.
\end{claim1}

\begin{proof}
Suppose that $d_{I_R}(a)\geq 2$. Let $x,y$ be two arbitrary vertices
in $N_{I_R}(a)$. If $xy\in E(G)$, then $\{a,x,y,a'_1,a'_2,a'_3\}$
induces a $Z_3$ in $G$. Note that $a'_1$ is not heavy in $G$. Thus
$x$ and $a'_3$ are heavy in $G$. Note that $x$ and $a'_3$ are
dissociated, a contradiction. This implies that $N_{I_R}(a)$ is an
independent set.

Since $\{a,x,y,a'_1\}$ induces a claw in $G$, and $\{a'_1,x\}$,
$\{a'_1,y\}$ are not heavy pairs of $G$ by Lemma \ref{LeClosed} (3),
we have $\{x,y\}$ is a heavy pair of $G$. We assume without loss of
generality that $x$ is heavy in $G$.

If $a$ is also heavy in $G$, then by Claim \ref{ClCommon}, $a,x$
have a common neighbor in $I_R$, contradicting the fact that
$N_{I_R}(a)$ is an independent set. So we conclude that $a$ is not
heavy in $G$.

Since $\{x,y\}$ is a heavy pair of $G$, by Lemma \ref{LeRegion} (4),
$x,y$ have two common neighbors in $I_R$. Let $x',y'$ be two
vertices in $N_{I_R}(x)\cap N_{I_R}(y)$. Clearly $ax',ay'\notin
E(G)$. If $x'y'\in E(G)$, then $\{x,x',y',a,a'_1,a'_2\}$ induces a
$Z_3$ in $G$. Since $a$ is light, $x',a'_2$ are heavy. Note that
$x'$ and $a'_2$ are dissociated, a contradiction. Thus we obtain
that $x'y'\notin E(G)$.

Note that $\{x,x',y',a\}$ induces a claw in $G$, and $a$ is light in
$G$. So one vertex of $\{x',y'\}$, say $x'$, is heavy in $G$. By
Claim \ref{ClCommon}, $x,x'$ have a common neighbor $x''$ in $I_R$.
Clearly $ax''\notin E(G)$. Thus $\{x,x',x'',a,a'_1,a'_2\}$ induces a
$Z_3$. Since $a$ is not heavy in $G$, $x',a'_2$ are heavy in $G$, a
contradiction.\qed
\end{proof}

Now let $N_{I_R}(a)=\{x\}$.

\begin{claim1}\label{ClAdjacent}
$N_R(a)=V(R)\backslash\{a\}$.
\end{claim1}

\begin{proof}
Suppose that $V(R)\backslash(\{a\}\cup N_R(a))\neq\emptyset$. By
Lemma \ref{LeRegion} (1), $R-x$ is connected. Let $y$ be a vertex in
$V(R)\backslash(\{a\}\cup N_R(a))$ such that $a,y$ have a common
neighbor $z$ in $R-x$. Note that $z$ is a frontier vertex of $R$.
Let $z'$ be a vertex in $N_{G-R}(z)$. Then $\{z,y,a,z'\}$ induces a
claw in $G$. Since $\{a,z'\}$, $\{y,z'\}$ are not heavy pairs of
$G$, $\{a,y\}$ is a heavy pair of $G$. By Lemma \ref{LeRegion} (4),
$a,y$ have two common neighbors in $I_R$, contradicting Claim
\ref{ClInterior}.\qed
\end{proof}

By Claims \ref{ClInterior} and \ref{ClAdjacent}, we can see that
$|I_R|=1$. Recall that one edge of $\{ab,bc,ac\}$ is not in $E(G)$.
By Claim \ref{ClAdjacent}, $ab,ac\in E(G)$. This implies that
$bc\notin E(G)$, and $\{a,b,c,a'_1\}$ induces a claw in $G$. Since
$\{b,a'_1\}$, $\{c,a'_1\}$ are not heavy pairs of $G$, $\{b,c\}$ is
a heavy pair of $G$. By Lemma \ref{LeRegion} (4), $b$ and $c$ have
two common neighbors in $I_R$, contradicting the fact that
$|I_R|=1$.\qed
\end{proof}

Following \cite{Br}, we  define $\mathcal{P}$ to be the class of
graphs obtained by taking two vertex-disjoint triangles
$a_1a_2a_3a_1$, $b_1b_2b_3b_1$ and by joining every pair of vertices
$\{a_i,b_i\}$ by a path $P_{k_i}=a_ic_i^1c_i^2\cdots
c_i^{k_i-2}b_i$, for $k_i\geq 3$ or by a triangle $a_ib_ic_ia_i$. We
denote the graphs in $\mathcal{P}$ by $P_{l_1,l_2,l_3}$, where
$l_i=k_i$ if $a_i,b_i$ are joined by a path $P_{k_i}$, and $l_i=T$
if $a_i,b_i$ are joined by a triangle. Note that $L_1=P_{T,T,T}$ and
$L_2=P_{3,T,T}$.

\begin{theorem}\label{ThBr}
\emph{({Brousek} \cite{Br})} Every non-hamiltonian 2-connected
claw-free graph contains an induced subgraph $H\in \mathcal{P}$.
\end{theorem}

\section{Proof of Theorem \ref{thNZL}}
Let $G'=cl_o(G)$. If $G'$ is hamiltonian, then so is $G$ by Theorem
\ref{ThC}, and we are done. Now we assume that $G'$ is not
hamiltonian. By Theorem \ref{ThBr}, $G'$ contains an induced
subgraph $H=P_{l_1,l_2,l_3}\in \mathcal{P}$. We denote the vertices
of $H$ by $a_i,b_i,c_i$ and $c_i^j$ as in Section 2. By Lemma
\ref{LeCenter}, $G'$ is $Z_3$-center-heavy.

\setcounter{claim1}{0}
\begin{claim1}
For $i\in\{1,2,3\}$, $l_i=3$ or $T$; and at most one of
$\{l_1,l_2,l_3\}$ is 3.
\end{claim1}

\begin{proof}
If one of $\{l_1,l_2,l_3\}$ is at least 4, say $l_1\geq 4$, then the
subgraph of $G'$ induced by $\{a_1,a_2,a_3,c_1^1,c_1^2,c_1^3\}$ is a
$Z_3$ (we set $c_1^3=b_1$ if $l_1=4$). Thus $c_1^1$ is heavy in
$G'$. If $l_2=T$, then the subgraph of $G'$ induced by
$\{a_2,a_1,a_3,b_2,b_1,c_1^{l_1-2}\}$ is a $Z_3$, implying $b_2$ is
heavy in $G'$. But $c_1^1$ and $b_2$ are dissociated, a
contradiction. If $l_2\neq T$, then the subgraph of $G'$ induced by
$\{a_2,a_1,a_3,c_2^1,\ldots,c_2^{l_2-2},b_2,b_1\}$ is a $Z_r$ with
$r\geq 3$, implying $c_2^1$ is heavy in $G'$. But $c_1^1$ and
$c_2^1$ are dissociated, a contradiction again. Thus we conclude
that $l_i=3$ or $T$ for all $i=1,2,3$.

If two of $\{l_1,l_2,l_3\}$ equal 3, say $l_1=l_2=3$, then the
subgraphs of $G'$ induced by $\{a_1,a_2,a_3,c_1^1,b_1,b_2\}$ and by
$\{a_2,a_1,a_3,c_2^1,b_2,b_1\}$ are $Z_3$'s. This implies that
$c_1^1$ and $c_2^1$ are heavy in $G'$. But $c_1^1$ and $c_2^1$ are
dissociated, a contradiction. Thus we conclude that at most one of
$\{l_1,l_2,l_3\}$ is 3.\qed
\end{proof}

By Claim 1, we assume without loss of generality that $l_2=l_3=T$
and $l_1=3$ or $T$. If $G'$ has only the nine vertices in $H$, then
$G'=L_1$ or $L_2$, and $G$ has no \emph{o}-eligible vertices. This
implies that $G=L_1$ or $L_2$. Now we assume that $G'$ has a tenth
vertex.

Let $A$ be the region containing $\{a_1,a_2,a_3\}$ and $B$ be the
region containing $\{b_1,b_2,$ $b_3\}$. For $l_i=T$, let $C_i$ be
the region containing $\{a_i,b_i,c_i\}$; and if $l_1=3$, then let
$C_1^1$ and $C_1^2$ be the regions containing $\{a_1,c_1^1\}$ and
$\{b_1,c_1^1\}$, respectively.

\begin{claim1}
$|V(A)|=|V(B)|=|V(C_i)|=3$; and if $l_1=3$, then
$|V(C_1^1)|=|V(C_1^2)|=2$.
\end{claim1}

\begin{proof}
Suppose that $|V(A)|\geq 4$. Let $x$ be a vertex in
$V(A)\backslash\{a_1,a_2,a_3\}$. Then the subgraphs of $G'$ induced
by $\{a_2,a_1,x,b_2,b_3,c_3\}$ and by $\{a_3,a_1,x,b_3,b_2,c_2\}$
are $Z_3$'s. This implies that $b_2$ and $b_3$ are heavy in $G'$.
Since there are two vertices $a_1,x$ nonadjacent to $b_2$ and $b_3$,
$b_2$ and $b_3$ have at least two common neighbors in $G'$. Let $y$
be a common neighbor of $b_2$ and $b_3$ in $G'$ other than $b_1$.
Then $y\in V(B)$, and the subgraphs of $G'$ induced by
$\{b_2,b_1,y,a_2,a_3,c_3\}$ is a $Z_3$. Thus $a_2$ is heavy in $G'$.
By Lemma \ref{LeClosed} (3), $a_2$ and $b_3$ are associated, a
contradiction. Thus we conclude that $|V(A)|=3$, and similarly,
$|V(B)|=3$.

Suppose that $|V(C_i)|\geq 4$ for $l_i=T$. We assume up to symmetry
that $|V(C_2)|\geq 4$. Let $x$ be a vertex in
$V(C_2)\backslash\{a_2,b_2,c_2\}$. Then the subgraph of $G'$ induced
by $\{a_2,c_2,x,a_3,b_3,b_1\}$ is a $Z_3$, implying that $a_3$ is
heavy in $G$. If $l_1=T$, then the subgraph of $G'$ induced by
$\{b_2,c_2,x,b_1,a_1,a_3\}$ is a $Z_3$; if $l_1=3$, then the
subgraph of $G'$ induced by $\{b_2,c_2,x,b_1,c_1,a_1\}$ is a $Z_3$.
In any case, we have $b_1$ is heavy in $G'$. But $a_3$ and $b_1$ are
dissociated in $G$, a contradiction.

Suppose that $l_1=3$ and $|V(C_1^1)|\geq3$. Let $x$ be a vertex in
$V(C_1^1)\backslash\{a_1,c_1^1\}$. Then the subgraphs of $G'$
induced by $\{a_1,c_1^1,x,a_2,b_2,b_3\}$ and by
$\{c_1^1,a_1,x,b_1,b_2,c_2\}$ are $Z_3$'s. This implies that $a_2$
and $b_1$ are heavy in $G'$. But $a_2$ and $b_1$ are dissociated, a
contradiction. Thus we conclude that $|V(C_1^1)|=2$, and similarly,
$|V(C_1^2)|=2$.\qed
\end{proof}

In the following, we set $S=\{v\in V(G'): N_{G'}(v)\cap
V(H)\neq\emptyset\}$.

\begin{claim1}
$l_1=3$, and for $x\in S$, $xc_2,xc_3\in E(G')$.
\end{claim1}

\begin{proof}
By Claim 2, all the neighbors of $a_1,a_2,a_3,b_1,b_2,b_3$ and
$c_1^1$ (if $l_1=3$) are in $H$. Note that $G'$ has at least 10
vertices. The vertices $a_1,a_2,a_3,b_1,b_2,b_3$ and $c_1^1$ (if
$l_1=3$) are not heavy in $G'$.

Let $x$ be a vertex in $S$. Suppose that $l_1=T$. Note that $x$
cannot be adjacent to all the three vertices $c_1,c_2,c_3$. We
assume up to symmetry that $xc_1\in E(G')$ and $xc_2\notin E(G')$.
Then the subgraph of $G'$ induced by $\{a_2,b_2,c_2,a_1,c_1,x\}$ is
a $Z_3$, implying $a_1$ is heavy in $G'$, a contradiction. Thus we
conclude that $l_1=3$.

Suppose that one edge of $xc_2,xc_3$ is not in $E(G')$, say
$xc_2\notin E(G')$. Then the subgraph of $G'$ induced by
$\{a_2,b_2,c_2,a_3,c_3,x\}$ is a $Z_3$, implying $a_3$ is heavy in
$G'$, a contradiction. Thus we conclude that $xc_2,xc_3\in
E(G')$.\qed
\end{proof}

Let $x$ be a vertex in $S$. By Claim 3, $xc_2,xc_3\in E(G')$. If
$G'$ has only ten vertices, then
$C=a_1a_2a_3c_3xc_2b_2b_3b_1c_1^1a_1$ is a Hamilton cycle of $G'$, a
contradiction. Suppose now that $G'$ has an eleventh vertex. Since
$G'$ is 2-connected, let $x'$ be a vertex in $S\backslash\{x\}$. By
Claim 3, $x'c_2,x'c_3\in E(G')$. Thus $xx'\in E(G')$. Note that
$N_{G'}(x)$ is neither a clique nor a disjoint union of two cliques
of $G'$. This implies that $x$ is an \emph{o}-eligible vertex of
$G'$, a contradiction.

The proof is complete.\qed




\end{document}